\documentclass{amsart}

\usepackage{amsmath,amssymb,amsthm,amsfonts,mathrsfs}
\usepackage{amscd,stmaryrd}
\usepackage{mathtools}
\usepackage[usenames,dvipsnames]{color}
\usepackage{hyperref}
\usepackage{graphicx,subfigure}
\definecolor{nicegreen}{RGB}{0,180,0}
\hypersetup{colorlinks=true,citecolor=nicegreen,linkcolor=Red,urlcolor=blue,pdfstartview=FitH}
\usepackage[divide={2.45cm,*,2.45cm}]{geometry}
\usepackage{enumerate}
\usepackage{color}
\usepackage[all]{xy}
\usepackage{tikz-cd}
\usepackage{wasysym}

\allowdisplaybreaks

\newtheorem{thm}{Theorem}[section]
\newtheorem*{thm*}{Theorem}
\newtheorem{cor}[thm]{Corollary}

\newtheorem{lemma}[thm]{Lemma}

\newtheorem*{propn*}{Proposition}

\theoremstyle{definition}

\newtheorem{rmk}[thm]{Remark}

\newcommand{\triv}{\textnormal{triv}}
\newcommand{\sign}{\textnormal{sign}}

\newcommand{\soc}{\textnormal{soc}}
\newcommand{\Hom}{\textnormal{Hom}}
\newcommand{\End}{\textnormal{End}}
\newcommand{\Ind}{\textnormal{Ind}}
\newcommand{\cind}{\textnormal{c-Ind}}
\newcommand{\Ext}{\textnormal{Ext}}

\newcommand{\ord}{\textnormal{Ord}}

\newcommand*{\longhookrightarrow}{\ensuremath{\lhook\joinrel\relbar\joinrel\rightarrow}}
\newcommand*{\longtwoheadrightarrow}{\ensuremath{\relbar\joinrel\twoheadrightarrow}}

\newcommand{\sm}[4]{\left(\begin{smallmatrix} #1 & #2 \\ #3 & #4 \end{smallmatrix}\right)}

\newcommand{\cC}{{\mathcal{C}}}

\newcommand{\cH}{{\mathcal{H}}}

\newcommand{\cR}{{\mathcal{R}}}

\newcommand{\cT}{{\mathcal{T}}}

\newcommand{\bbC}{{\mathbb{C}}}

\newcommand{\bbF}{{\mathbb{F}}}

\newcommand{\bbQ}{{\mathbb{Q}}}

\newcommand{\bbZ}{{\mathbb{Z}}}

\newcommand{\fJ}{{\mathfrak{J}}}

\begin{document}
\nocite{}

\title{Derived right adjoints of parabolic induction: an example}
\date{}
\author{Karol Kozio\l}
\thanks{The author was supported by NSF Grant DMS-2101836.}
\address{Baruch College, City University of New York, 1 Bernard Baruch Way, New York, NY 10010} \email{karol.koziol@baruch.cuny.edu}
\thanks{We thank Ramla Abdellatif and the anonymous referee for useful comments and discussion.}


\begin{abstract}
Suppose $p \geq 5$ is a prime number, and let $G = \textnormal{SL}_2(\mathbb{Q}_p)$.  We calculate the derived functors $\textnormal{R}^n\mathcal{R}_B^G(\pi)$, where $B$ is a Borel subgroup of $G$, $\mathcal{R}_B^G$ is the right adjoint of smooth parabolic induction constructed by Vign\'eras, and $\pi$ is any smooth, absolutely irreducible, mod $p$ representation of $G$.
\end{abstract}

\maketitle

\section{Introduction}

One of the most fundamental operations in the representation theory of $p$-adic groups is that of \textbf{parabolic induction}: given a $p$-adic connected reductive group $G$, a rational parabolic subgroup $P = MN$, and a smooth representation $\sigma$ of the Levi quotient $M$ of $P$ (over some coefficient field $C$), we can construct the $G$-representation $\Ind_P^G(\sigma)$ induced from (the inflation to $P$ of) $\sigma$.  In this way, we obtain an exact functor $\Ind_P^G$ from the category of smooth $M$-representations over $C$ to the category of smooth $G$-representations over $C$.  The relevance of this functor comes from the fact that any smooth irreducible admissible $G$-representation arises as a subquotient of some $\Ind_P^G(\sigma)$, where $\sigma$ is a so-called supercuspidal representation.

In \cite{vigneras:rightadj}, Vign\'eras shows that the functor $\Ind_P^G$ commutes with small direct sums.  Since the category of smooth $M$-representations is a locally small, Grothendieck abelian category, a version of the Adjoint Functor Theorem implies that $\Ind_P^G$ possesses a (left-exact) right adjoint $\cR_P^G$.  When the characteristic of $C$ is different from $p$, we can identify this functor explicitly: $\cR_P^G$ is naturally isomorphic to the normalized Jacquet module relative to the opposite parabolic subgroup $P^{-} = MN^{-}$.  For $C = \bbC$, this is \cite[Main Thm.]{bernstein:secondadjoint} and \cite[Thm. 3]{bushnell}, and the general $\textnormal{char}(C) \neq p$ case follows from \cite[Cor. 1.3]{DHKM:heckealgs} (see also \cite[Thm. 1.5]{dat} for earlier partial results).

When the characteristic of $C$ is equal to $p$, the functor $\cR_P^G$ is more mysterious.  When restricted to the category of admissible representations, $\cR_P^G$ is isomorphic to Emerton's functor $\ord_{P^-}^G$ of ordinary parts (see \cite[Cor. 4.13]{AHV} for this equivalence).  However, it is not clear how to describe $\cR_P^G$ on the entire category of smooth representations.

There have also been fascinating recent advances in derived aspects of the mod $p$ representation theory of $p$-adic reductive groups, which provide new methods for approaching the mod $p$ Langlands program (see \cite{schneider:dga}, \cite{harris:specs}).  In particular, in the article \cite{scherotzkeschneider} the authors consider the total derived functors $\textnormal{R}\cR_P^G$ over a field of characteristic $p$, and leave open the question of explicitly calculating these derived functors.  The main result of this note is the following, which addresses some of these questions.

\begin{thm*}
Suppose $\pi$ is a smooth, absolutely irreducible representation of $G = \textnormal{SL}_2(\bbQ_p)$ over a field $C$ of characteristic $p \geq 5$, and let $B$ denote the upper triangular Borel subgroup of $G$.  Then, for all $n \geq 0$, we have an isomorphism of smooth admissible $T$-representations
$$\textnormal{R}^n\cR_B^G(\pi) \cong \textnormal{R}^n\ord_{B^-}^G(\pi).$$
\end{thm*}

We make some comments regarding this result.  As the category of admissible $G$-representations does not have enough injectives, the derived functors $\textnormal{R}^n\ord_{B^-}^G$ must be computed in the category of \emph{locally} admissible $G$-representations.  However, it is not known whether $\ord_{B^-}^G$ agrees with $\cR_B^G$ on this category, and therefore we cannot construct a direct comparison between $\textnormal{R}^n\cR_B^G$ and $\textnormal{R}^n\ord_{B^-}^G$.  To address this subtlety, we use the results of \cite{koziol:functorial} to relate $\cR_B^G$ to the analogously defined functor on pro-$p$-Iwahori--Hecke modules, where we can do explicit calculations to evaluate $\textnormal{R}^n\cR_B^G(\pi)$ (and deduce \textit{a posteriori} the isomorphism with $\textnormal{R}^n\ord_{B^-}^G(\pi)$).  To our knowledge, the results computing $\textnormal{R}^n\cR_B^G(\pi)$ are the first of their kind.

We are hopeful that some of the techniques for calculating $\textnormal{R}^n\cR_B^G(\pi)$ will generalize to other $p$-adic reductive groups (though some are specific to the group $\textnormal{SL}_2(\bbQ_p)$).

\section{Preparation}

\subsection{}
Suppose $p \geq 5$ is a prime number, and define $G := \textnormal{SL}_2(\bbQ_p)$.  We let $B$ denote the upper triangular Borel subgroup, and $T$ the diagonal maximal torus.  We let $\mathfrak{Rep}^\infty(-)$ denote the category of smooth representations with coefficients in a field $C$ of characteristic $p$.  We will examine the (exact) functor of parabolic induction
$$\textnormal{Ind}_B^G: \mathfrak{Rep}^\infty(T) \longrightarrow \mathfrak{Rep}^\infty(G),$$
and its right adjoint
$$\cR_B^G: \mathfrak{Rep}^\infty(G) \longrightarrow \mathfrak{Rep}^\infty(T),$$
constructed in \cite[\S 4]{vigneras:rightadj}.

We also define two distinguished characters $\overline{\rho}, \overline{\alpha}:T \longrightarrow C^\times$ which we use in the sequel:
$$\overline{\rho}\left(\begin{pmatrix} xp^a & 0 \\ 0 & x^{-1}p^{-a}\end{pmatrix}\right) = \overline{x},\qquad \overline{\alpha}\left(\begin{pmatrix} xp^a & 0 \\ 0 & x^{-1}p^{-a}\end{pmatrix}\right) = \overline{x}^2,$$
where $a \in \bbZ$, $x\in \bbZ_p^\times$, and where $\overline{x} \in \bbF_p^\times \subset C^\times$ denotes the mod $p$ reduction of $x$.

\subsection{}
In order to understand the right derived functors $\textnormal{R}^n\cR_B^G$, we use pro-$p$-Iwahori--Hecke algebras.  Let $I_1$ denote the subgroup of $\textnormal{SL}_2(\bbZ_p)$ which is upper triangular and unipotent modulo $p$, and let $T_1 := T \cap I_1$.  Note that $T_1 \cong 1 + p\bbZ_p \cong \bbZ_p$, so that $T_1$ has cohomological dimension 1.  We let $\cH$ denote the pro-$p$-Iwahori--Hecke algebra of $G$ with respect to $I_1$, and let $\cH_T$ denote the pro-$p$-Iwahori--Hecke algebra of $T$ with respect to $T_1$ (see \cite[\S 4]{vigneras:hecke1} for more details and definitions).  We note that $\cH_T \cong C[T/T_1] \cong C[\bbZ] \otimes_C C[\bbF_p^\times]$, and therefore $\cH_T$ has global dimension 1.

We have analogous functors of parabolic induction
$$\textnormal{Ind}_{\cH_T}^\cH: \mathfrak{Mod-}\cH_T \longrightarrow \mathfrak{Mod-}\cH$$
and its right adjoint
$$\cR_{\cH_T}^\cH: \mathfrak{Mod-}\cH \longrightarrow \mathfrak{Mod-}\cH_T,$$
defined on categories of right modules.  We refer to \cite[\S 4.2]{olliviervigneras} for details and definitions.

\subsection{}

Given a smooth representation $\pi$ of $G$, the space $\pi^{I_1}$ of $I_1$-invariants has a right action of $\cH$, recalled in \cite[pf. of Lem. 4.5]{olliviervigneras}.  Passing to derived functors, the cohomology spaces $\textnormal{H}^i(I_1,\pi)$ also inherit a right action of $\cH$ (described in \cite[\S 2.3]{koziol:functorial}).  We have analogous constructions for $T$ and $\cH_T$.

Given a smooth character $\chi:T\longrightarrow C^\times$, we have $\chi^{T_1} = \chi$, and therefore $\chi$ inherits the structure of a right $\cH_T$-module.  We will use the same notation $\chi$ to denote the $T$-representation and the resulting $\cH_T$-module; the meaning should be clear from context.

\subsection{}
The goal will be to compute $\textnormal{R}^{n}\cR^G_B(\pi)$ where $\pi$ is an absolutely irreducible admissible $G$-representation.  Our main tool will be \cite[Thm. 3.13]{koziol:functorial}: if $\pi$ is an admissible $G$-representation, then we have an $E_2$ spectral sequence of $\cH_T$-modules
$$\textnormal{H}^i\big(T_1, \textnormal{R}^j\cR_B^G(\pi)\big) \Longrightarrow \cR^{\cH}_{\cH_T}\big(\textnormal{H}^{i + j}(I_1, \pi)\big).$$
(Note that the assumption in \cite[\S 3]{koziol:functorial} that $C$ be finite is not required for the construction of the above spectral sequence.)  Since $T_1$ has cohomological dimension 1, the above spectral sequence degenerates at the $E_2$ page to give an isomorphism
\begin{equation}\label{deg1}
\cR_{\cH_T}^{\cH}(\pi^{I_1}) \cong \cR_B^G(\pi)^{T_1}
\end{equation}
and, for $n \geq 1$, an exact sequence
\begin{equation}\label{deg2}
0 \longrightarrow \textnormal{H}^1(T_1, \textnormal{R}^{n - 1}\cR^G_B(\pi)) \longrightarrow \cR_{\cH_T}^{\cH}(\textnormal{H}^n(I_1,\pi)) \longrightarrow (\textnormal{R}^n\cR^G_B(\pi))^{T_1} \longrightarrow 0.
\end{equation}
In particular, this implies the following:
\begin{lemma}
If $\pi$ is an admissible $G$-representation, then each $\textnormal{R}^n\cR_B^G(\pi)$ is an admissible $T$-representation.
\end{lemma}

\begin{proof}
Assuming $\pi$ is admissible, \cite[Lem. 3.4.4]{emerton:ord2} implies that $\textnormal{H}^n(I_1,\pi)$ is finite-dimensional for all $n\geq 0$.  By \cite[\S 4.2, Property (4)]{olliviervigneras}, the space $\cR_{\cH_T}^{\cH}(\textnormal{H}^n(I_1,\pi))$ is also finite-dimensional, and therefore \cite[Thm. 6.3.2]{paskunas:diags}, along with \eqref{deg1} and \eqref{deg2}, imply that each $\textnormal{R}^n\cR_B^G(\pi)$ is an admissible $T$-representation.
\end{proof}

\subsection{}
In addition to the functor $\cR_B^G$, we will also use Emerton's functors of derived ordinary parts (see \cite{emerton:ord1}, \cite{emerton:ord2} for the relevant definitions; we note that for our applications, the assumption in \textit{op. cit.} that the residue field of $A$ is finite is not necessary).  We recall that the functor $\ord_{B^-}^G$ is the right adjoint to the parabolic induction functor from $\mathfrak{Rep}^{\textnormal{ladm}}(T)$, the category of locally admissible $T$-representations, to $\mathfrak{Rep}^{\textnormal{ladm}}(G)$, the category of locally admissible $G$-representations.

For a locally admissible $T$-representation $\sigma$ and a locally admissible $G$-representation $\pi$, we have two spectral sequences relating extensions and parabolic induction: one in the smooth category
\begin{equation}\label{ord1}
\Ext_T^i(\sigma,\textnormal{R}^j\cR_B^G(\pi)) \Longrightarrow \Ext_G^{i + j}(\Ind_B^G(\sigma),\pi),
\end{equation}
(coming from an application of the Grothendieck spectral sequence) and one in the locally admissible category
\begin{equation}\label{ord2}
\Ext_{T,\textnormal{ladm}}^i(\sigma,\textnormal{R}^j\ord_{B^-}^G(\pi)) \Longrightarrow \Ext_{G,\textnormal{ladm}}^{i + j}(\Ind_B^G(\sigma),\pi)
\end{equation}
(cf. \cite[Eq. (3.7.4)]{emerton:ord2}).  By Corollary \ref{paskunascor} below, we have $\Ext_{H,\textnormal{ladm}}^n(\tau,\pi) \cong \Ext_H^n(\tau,\pi)$ if $\tau$ and $\pi$ are locally admissible $H$-representations, where $H \in \{T, G\}$.  In addition, by \cite[Thm. A.4]{koziol:functorial}, we have $\textnormal{R}^j\ord_{B^-}^G \simeq \textnormal{H}^j\ord_{B^-}^G$ for the group $\textnormal{SL}_2$, and the latter can be explicitly calculated by work of Emerton and Hauseux (\cite{emerton:ord2}, \cite{hauseux}).  Therefore, the spectral sequence \eqref{ord2} becomes
\begin{equation}\label{ord3}
\Ext_{T}^i(\sigma,\textnormal{H}^j\ord_{B^-}^G(\pi)) \Longrightarrow \Ext_{G}^{i + j}(\Ind_B^G(\sigma),\pi).
\end{equation}

\section{Locally admissible representations vs. smooth representations}

In this section, we compare locally admissible and smooth representations.

\begin{lemma}
Let $H \in \{T, G\}$, let $\iota: \mathfrak{Rep}^{\textnormal{ladm}}(H) \longrightarrow \mathfrak{Rep}^\infty(H)$ denote the fully faithful inclusion, and let $\fJ \in \mathfrak{Rep}^{\textnormal{ladm}}(H)$ denote an injective object.  Then $\iota(\fJ) \in \mathfrak{Rep}^\infty(H)$ is also an injective object.  
\end{lemma}

\begin{proof}
The proofs are virtually identical to the proof of \cite[Prop. 5.16]{paskunas:montrealfunctor}.  Since we do not fix the central character, we briefly highlight the main differences, and explain why the arguments work in our setting.  Once again, the assumption in \textit{op. cit.} that the coefficient field $k$ is finite is not necessary for our purposes.

Suppose first that $H = T$.  We outline the main changes to the proof in \emph{loc. cit.} (using that article's notation):
\begin{itemize}
\item We replace $\cind_{KZ}^G$ with $\cind_{T_0}^T$ in the proof, where $T_0 \cong \bbZ_p^\times$ denotes the maximal compact subgroup of $T$.  
\item Since $\bbF_p \subset C$, we may take the representation $\sigma$ to be a character; further, by twisting, we may assume this character is trivial.  
\item Instead of \cite[Cor. 3.8]{emertonpaskunas}, we use \cite[Thm. A.8]{koziol:functorial}.
\item The reference to \cite[Prop. 18]{barthellivne} is no longer applicable; however, in this case, we still have a non-injective $T$-equivariant surjection $\psi:\cind_{T_0}^T(\mathbf{1}_{T_0}) \cong C[X^{\pm 1}] \longtwoheadrightarrow A$.  Since this map is equivariant for the action of $C[T/T_0] \cong C[X^{\pm 1}]$ on both sides, the structure theorem for finitely generated modules over a PID implies that $A \cong C[X^{\pm 1}]/(f(X))$, where $f(X)$ is some nonzero polynomial.  In particular, $A$ is finite-dimensional and admissible.  
\item We may still appeal to \cite[Thm. 2.3.8]{emerton:ord1} to obtain the equivalence of the notions of ``locally admissible'' and ``locally finite'' representations.
\end{itemize}
The other parts of the proof apply with minor changes to give the desired result.

Suppose now that $H = G = \textnormal{SL}_2(\bbQ_p)$.  Since the center $Z$ is a finite group of order 2, and since $p \geq 5$, the action of $Z$ on any smooth representation is semisimple.  In particular, we may assume without loss of generality that $Z$ acts by a character on $\fJ$.  The proof contained in \cite{paskunas:montrealfunctor} may again be adapted to the group $\textnormal{SL}_2(\bbQ_p)$, with the following changes:
\begin{itemize}
\item We replace $\cind_{KZ}^G$ with $\cind_{\textnormal{SL}_2(\bbZ_p)}^{\textnormal{SL}_2(\bbQ_p)}$ in the proof.  
\item Instead of \cite[Cor. 3.8]{emertonpaskunas}, we use \cite[\S A.2, pf. of Step 1]{koziol:functorial}.
\item The analog of \cite[Prop. 18]{barthellivne} holds for the group $\textnormal{SL}_2(\bbQ_p)$ (see the remark below).
\item Let $\cT \in \End_{\textnormal{SL}_2(\bbQ_p)}(\cind_{\textnormal{SL}_2(\bbZ_p)}^{\textnormal{SL}_2(\bbQ_p)}(\sigma))$ denote the spherical Hecke operator associated to the $\textnormal{SL}_2(\bbZ_p)$-bi-equivariant function with support $\textnormal{SL}_2(\bbZ_p) \sm{p}{0}{0}{p^{-1}} \textnormal{SL}_2(\bbZ_p)$, so that the relevant spherical Hecke algebra is a polynomial algebra in $\cT$.  To see that $\cind_{\textnormal{SL}_2(\bbZ_p)}^{\textnormal{SL}_2(\bbQ_p)}(\sigma)/ (\cT - \lambda)$ is of finite length, we appeal to \cite[Thm. 3.18, Prop. 2.7, Cor. 3.26(3), pf. of Prop. 4.4]{abdellatif}.  Further, each representation $\cind_{\textnormal{SL}_2(\bbZ_p)}^{\textnormal{SL}_2(\bbQ_p)}(\sigma)/ (\cT - \lambda)$ is admissible by \cite[Prop. 2.9, Cor. 3.26(3), Cor. 4.6, Prop. 4.7]{abdellatif}.  Therefore, the representation $A$ appearing in \cite[Pf. of Prop. 5.16]{paskunas:montrealfunctor} is admissible.
\item In order to prove the equivalence of ``locally admissible'' and ``locally finite,'' we need an analog of \cite[Thm. 2.3.8]{emerton:ord1} for the group $\textnormal{SL}_2(\bbQ_p)$ (more precisely, we need a version of the equivalence $(1) \Longleftrightarrow (2)$ of \textit{loc. cit.}, Lemma 2.3.6).  The validity of this claim follows from the claim that the quotients $\cind_{\textnormal{SL}_2(\bbZ_p)}^{\textnormal{SL}_2(\bbQ_p)}(\sigma)/ (\cT - \lambda)$ are of finite length, proved in the previous bullet point.  (Of independent interest: in order to verify the equivalence $(1) \Longleftrightarrow (3)$ of \cite[Lem, 2.3.6]{emerton:ord1} for $\textnormal{SL}_2(\bbQ_p)$, one can use \cite[Thm. 3.36(3)]{abdellatif}.)
\end{itemize}
The other parts of the proof apply with minor changes to give the desired result.  
\end{proof}

\begin{rmk}
We believe the counterexample contained in \cite[\S 3.7.3]{abdellatif} is incorrect\footnote{We have discussed this with Abdellatif, and she has agreed with our assessment.}.  In particular, in the notation of \emph{op. cit.}, the Hecke algebra $\cH(G_S,I_S,\omega_{(p - 1)/2})$ contains functions supported on double cosets of the form $I_S\sm{0}{-p^{-\ell}}{p^\ell}{0}I_S$.  Therefore, there are more Hecke operators contained in $\cH(G_S, I_S, \omega_{(p - 1)/2})$ than simply the span of the $T_{2n, 2n + 1}$.  Moreover, the action of the Hecke operator $T_{\sm{0}{1}{-1}{0}}$ on the function $f_1$ is given by $f_1 \cdot T_{\sm{0}{1}{-1}{0}} = (-1)^{(p - 1)/2}f_{-1}$, which implies that the submodule $\bigoplus_{n \geq 1} \overline{\bbF}_p f_n$ is not Hecke-stable.

After checking details, we believe that the analog of \cite[Prop. 18]{barthellivne} for the group $\textnormal{SL}_2(\bbQ_p)$ is indeed true: any $\cH$-stable submodule $W \subset (\cind_{\textnormal{SL}_2(\bbZ_p)}^{\textnormal{SL}_2(\bbQ_p)}(\textnormal{Sym}^r))^{I_1}$ is of finite codimension.  
\end{rmk}

\begin{cor}
\label{paskunascor}
Let $H \in \{T, G\}$, and let $\tau, \pi$ be two locally admissible $H$-representations.  Then
$$\Ext^i_{H,\textnormal{ladm}}(\tau, \pi) \cong \Ext^i_{H}(\tau, \pi).$$
\end{cor}

\section{Some Ext calculations}

We begin with some calculations of Ext groups which we will need below.

\subsection{$T$-extensions}\label{sec:T-exts}
Let $\sigma,\kappa$ denote two smooth $T$-representations, and suppose that $\sigma$ is generated by its space of $T_1$-invariant vectors.  (As $T_1$ is normal in $T$, this is equivalent to requiring that $\sigma^{T_1} = \sigma$.)  Since the category of smooth $T$-representations generated by their $T_1$-invariants is equivalent to the category of right $\cH_T$-modules, we get an $E_2$ spectral sequence of $C$-vector spaces
\begin{equation}\label{Tss}
\Ext_{\cH_T}^i(\sigma^{T_1}, \textnormal{H}^j(T_1,\kappa)) \Longrightarrow \Ext_T^{i + j}(\sigma,\kappa).
\end{equation}
Compare \cite[eq. (33)]{paskunas:exts}.

\begin{lemma}\label{T-exts-lemma}
Let $\chi, \chi':T \longrightarrow C^\times$ denote two smooth characters of $T$.  Then
$$\dim_C\big(\Ext_T^{n}(\chi,\chi')\big) = \binom{2}{n}\delta_{\chi,\chi'},$$
where $\delta_{\chi,\chi'}$ denotes the Kronecker delta function.
\end{lemma}

\begin{proof}
Since $\cH_T$ is of global dimension 1 and $T_1$ has cohomological dimension 1, the spectral sequence \eqref{Tss} degenerates at the $E_2$ page.  The equations
$$\textnormal{H}^1(T_1,\chi') \cong \chi'$$
and
$$\dim_C\big(\Ext_{\cH_T}^{n}(\chi,\chi')\big) = \binom{1}{n}\delta_{\chi,\chi'}$$
then give the result.  
\end{proof}

\begin{lemma}\label{Tfinlen}
Suppose $\sigma, \kappa$ are smooth $T$-representations, and suppose $\sigma$ has finite length.  Then, for $n \geq 3$, we have
$$\Ext_T^n(\sigma,\kappa) = 0.$$
\end{lemma}

\begin{proof}
By induction on length, it suffices to assume $\sigma$ is simple.  In particular, $\sigma$ is generated by its $T_1$-invariant vectors.  Since $\cH_T$ is of global dimension 1 and $T_1$ has cohomological dimension 1, the spectral sequence \eqref{Tss} degenerates at the $E_2$ page, which gives the result.
\end{proof}

Now let $\chi:T \longrightarrow C^\times$ be a smooth character, and let $\sigma$ denote a nonsplit extension of $\chi$ by $\chi$:
\begin{equation}\label{nonsplit-ses-sigma}
0 \longrightarrow \chi \longrightarrow \sigma \longrightarrow \chi \longrightarrow 0
\end{equation}

\begin{lemma}\label{extsigmachi}
We have
$$\dim_C\big(\Ext_T^n(\sigma, \chi)\big) = \binom{2}{n}.$$
\end{lemma}

\begin{proof}
The degree $n = 0$ case follows from the fact that $\sigma$ is a nonsplit extension, while the $n = 1$ case follows from a direct calculation with Yoneda extensions.  Applying the functor $\Hom_T(-,\chi)$ to the short exact sequence \eqref{nonsplit-ses-sigma} gives a long exact sequence of Ext groups; taking the Euler characteristic, using Lemmas \ref{T-exts-lemma} and \ref{Tfinlen} and the degrees already computed gives the result in degrees $n \geq 2$.  (Alternatively, one can dualize and use the spectral sequence \eqref{Tss}.)
\end{proof}

\section{Right adjoint calculations}

We are now in a position to calculate $\textnormal{R}^n\cR_B^G$.  The smooth, absolutely irreducible representations of $\textnormal{SL}_2(\bbQ_p)$ are divided into four classes (cf. \cite[Thm. 4]{henniartvigneras:descentp}, \cite[Thms. 3.42, 4.12]{abdellatif}), and we discuss each in turn.

\subsection{Principal series}
This is the most involved calculation.

For this entire subsection we let $\pi = \Ind_B^G(\chi)$, where $\chi:T \longrightarrow C^\times$ is a smooth character.  By \cite[\S 5.4.2]{koziol:functorial}, we have
\begin{equation}
\label{H1PS}
\textnormal{H}^n(I_1, \pi) = \begin{cases} \Ind_{\cH_T}^{\cH}(\chi) & \textnormal{if $n = 0$}, \\ \textnormal{extension of } \Ind_{\cH_T}^\cH(\chi^{-1}\overline{\alpha})^\vee \textnormal{ by } \Ind_{\cH_T}^{\cH}(\chi) & \textnormal{if $n = 1$}, \\  \Ind_{\cH_T}^\cH(\chi^{-1}\overline{\alpha})^\vee & \textnormal{if $n = 2$}, \\ 0 & \textnormal{if $n \geq 3$.} \end{cases}
\end{equation}
The extension in the $\textnormal{H}^1$ term is nonsplit if and only if $\chi = \overline{\rho}$, and for a right $\cH$-module $M$ we equip $M^\vee := \Hom_C(M,C)$ with the structure of a right $\cH$-module as in \cite[\S 4]{abe:involutions}.  We now apply $\cR_{\cH_T}^{\cH}$ to \eqref{H1PS}, and use \cite[Thm. 5.20]{abe:inductions} and \cite[Thm. 4.9]{abe:involutions}.  This gives
\begin{equation}
\label{H1PS-2}
\cR^{\cH}_{\cH_T}(\textnormal{H}^n(I_1, \pi)) = \begin{cases} \chi & \textnormal{if $n = 0$}, \\ \textnormal{extension of } \chi^{-1}\overline{\alpha} \textnormal{ by } \chi & \textnormal{if $n = 1$}, \\  \chi^{-1}\overline{\alpha} & \textnormal{if $n = 2$}, \\ 0 & \textnormal{if $n \geq 3$.} \end{cases}
\end{equation}
Once again, the degree 1 term is nonsplit if and only if $\chi = \overline{\rho}$.  We will use the above as input into the short exact sequence \eqref{deg2}.

Since the unit of the adjunction $(\Ind_B^G, \cR_B^G)$ is an isomorphism \cite[Thm. 5.3]{vigneras:rightadj}, we get
$$\cR_B^G(\pi) = \chi.$$
Suppose next that $n \geq 3$.  Then the short exact sequence \eqref{deg2} and equations \eqref{H1PS-2} imply $(\textnormal{R}^n\cR^G_B(\pi))^{T_1} = 0$, so that
$$\textnormal{R}^n\cR^G_B(\pi) = 0 \qquad \textnormal{for all $n\geq 3$.}$$

We now calculate the remaining two degrees.

\begin{lemma}\label{lemma-R2=0}
We have $\textnormal{R}^2\cR_B^G(\pi) = 0$.  
\end{lemma}

\begin{proof}
We proceed in several steps.  Note first that the short exact sequence \eqref{deg2} for $n = 2$ and $n = 3$ along with the equations \eqref{H1PS-2} give
$$\dim_C\left(\textnormal{R}^2\cR_B^G(\pi)^{T_1}\right) \leq 1, \qquad \dim_C\left(\textnormal{H}^1(T_1, \textnormal{R}^{2}\cR^G_B(\pi))\right) = 0.$$

\noindent \textit{Step 1. The $T_1$-representation $\textnormal{R}^{2}\cR^G_B(\pi)|_{T_1}$ is either $0$ or isomorphic to $\cC^\infty(T_1,C)$.}

This essentially follows from the dimension calculations above.  We elaborate.  Since $\textnormal{R}^2\cR_B^G(\pi)$ is an admissible $T$-representation, the Pontryagin dual $\textnormal{R}^2\cR_B^G(\pi)^\vee := \textnormal{Hom}_C(\textnormal{R}^2\cR_B^G(\pi),C)$ is a finitely generated module over the DVR $C[\![ T_1 ]\!] \cong C[\![ X ]\!]$.  Furthermore, the cosocle of $\textnormal{R}^2\cR_B^G(\pi)^\vee$ is dual to the space of invariants $\textnormal{R}^2\cR_B^G(\pi)^{T_1}$.  By Nakayama's lemma, $\textnormal{R}^2\cR_B^G(\pi)^\vee$ is generated by at most one element, which implies
$$\textnormal{R}^2\cR_B^G(\pi)^\vee \cong \qquad 0 \qquad \textnormal{or} \qquad C[\![ X ]\!] \qquad \textnormal{or} \qquad C[\![ X ]\!]/X^r$$
for some $r \geq 1$.  Dualizing, we obtain
$$\textnormal{R}^2\cR_B^G(\pi)|_{T_1} \cong \qquad  0 \qquad \textnormal{or} \qquad \cC^\infty(T_1, C) \qquad \textnormal{or} \qquad (C[\![ X ]\!]/X^r)^\vee.$$
However, the condition $\textnormal{H}^1(T_1, \textnormal{R}^{2}\cR^G_B(\pi)) = 0$ rules out the last possibility.

\vspace{10pt}

\noindent \textit{Step 2. We calculate the graded pieces of the socle filtration of $\textnormal{R}^{2}\cR^G_B(\pi)$.}

When $\textnormal{R}^{2}\cR^G_B(\pi) = 0$, there is nothing to prove, so assume the contrary.  For the sake of brevity, we set $\kappa := \textnormal{R}^{2}\cR^G_B(\pi)$, and for $i\geq 0$, let $\sigma^i := \textnormal{soc}^i_T(\kappa)$ denote the socle filtration of $\kappa$.

First, we have 
$$\sigma^1 = \textnormal{soc}_T(\kappa) \subset \kappa^{T_1} \cong \chi^{-1}\overline{\alpha}$$ 
(by \eqref{deg2} in degree $n = 2$, assuming $\kappa \neq 0$).  Thus, $\kappa^{T_1}$ is semisimple as a $T$-representation, so that $\kappa^{T_1} \subset \soc_T(\kappa)$, and this implies $\sigma^1 \cong \chi^{-1}\overline{\alpha}$.  Similarly, we have 
$$\sigma^2/\sigma^1 = \textnormal{soc}_T(\kappa/\sigma^1) = \textnormal{soc}_T(\kappa/\chi^{-1}\overline{\alpha}) \subset (\kappa/\chi^{-1}\overline{\alpha})^{T_1}.$$
To determine the latter space (which is nonzero by Step 1), we apply the functor of $T_1$-invariants to the short exact sequence
$$0 \longrightarrow \chi^{-1}\overline{\alpha} \longrightarrow \kappa \longrightarrow \kappa/\chi^{-1}\overline{\alpha} \longrightarrow 0$$
to get
$$0 \longrightarrow \chi^{-1}\overline{\alpha} \stackrel{\sim}{\longrightarrow} \kappa^{T_1} \stackrel{0}{\longrightarrow} (\kappa/\chi^{-1}\overline{\alpha})^{T_1} \longrightarrow \textnormal{H}^1(T_1,\chi^{-1}\overline{\alpha}) \cong \chi^{-1}\overline{\alpha}.$$
Thus, we obtain $(\kappa/\chi^{-1}\overline{\alpha})^{T_1} \cong \chi^{-1}\overline{\alpha}$, and consequently $\sigma^2/\sigma^1 \cong \chi^{-1}\overline{\alpha}$.  Continuing in this way, we see that $\dim_C(\sigma^i) = i$ and $\sigma^{i + 1}/\sigma^i \cong \chi^{-1}\overline{\alpha}$ for all $i \geq 0$.

\vspace{10pt}

\noindent \textit{Step 3. We have $\textnormal{R}^2\cR_B^G(\pi) = 0$.}

Assume the contrary, so that $\textnormal{R}^2\cR_B^G(\pi)|_{T_1} \cong \cC^\infty(T_1,C)$.  As in Step 2, we have $\textnormal{R}^2\cR_B^G(\pi)^{T_1} \cong \chi^{-1}\overline{\alpha}$.  By injectivity of $\cC^\infty(T_1,C)$, the spectral sequence \eqref{Tss} for $\kappa = \textnormal{R}^2\cR_B^G(\pi)$ collapses to give
\begin{equation}\label{R2dims}
\dim_C\big(\Ext_T^i(\chi^{-1}\overline{\alpha}, \textnormal{R}^2\cR_B^G(\pi))\big) = \dim_C\big(\Ext_{\cH_T}^i(\chi^{-1}\overline{\alpha}, \chi^{-1}\overline{\alpha})\big) = \binom{1}{i}.
\end{equation}

Let $\sigma = \sigma^2$ denote the second step in the socle filtration of $\textnormal{R}^2\cR_B^G(\pi)$, as in Step 2 above, and note that by construction we have $\dim_C(\Hom_T(\sigma,\textnormal{R}^2\cR_B^G(\pi))) = 2$.  Applying $\Hom_T(-,\textnormal{R}^2\cR_B^G(\pi))$ to the short exact sequence 
$$0 \longrightarrow \chi^{-1}\overline{\alpha} \longrightarrow \sigma \longrightarrow \chi^{-1}\overline{\alpha} \longrightarrow 0$$
gives a long exact sequence of Ext groups.  Taking the Euler characteristic and using \eqref{R2dims} implies
\begin{equation}\label{dimsigmaR2}
\dim_C\big(\Ext_T^1(\sigma,\textnormal{R}^2\cR_B^G(\pi))\big) = \dim_C\big(\Hom_T(\sigma,\textnormal{R}^2\cR_B^G(\pi))\big) = 2.
\end{equation}

We first examine the spectral sequence \eqref{ord3} for $\sigma$ and $\pi$ as above.  By Lemma \ref{Tfinlen} and the fact that $\textnormal{H}^j\ord_{B^-}^G(\pi) = 0$ for $j \geq 2$ (\cite[Prop. 3.6.1]{emerton:ord2}), we have $E_2^{i,j} = 0$ for $i \geq 3$ or $j \geq 2$.  This implies 
$$\Ext^3_G(\Ind_B^G(\sigma), \pi) \cong E_\infty^{2,1} = E_2^{2,1} = \Ext_T^2(\sigma,\textnormal{H}^1\ord_{B^-}^G(\pi)) \cong \Ext_T^2(\sigma,\chi^{-1}\overline{\alpha}),$$
where the last isomorphism follows from \cite[Cor. 3.3.8(ii)]{hauseux}.  Therefore, we get
\begin{equation}\label{Ext3dim=1}
\dim_C\big(\Ext^3_G(\Ind_B^G(\sigma),\pi)\big) = 1 
\end{equation}
by Lemma \ref{extsigmachi}.

Consider now the spectral sequence \eqref{ord1}.  By Lemma \ref{Tfinlen} and the $\textnormal{R}^n\cR_B^G$ already calculated, we have $E_2^{i,j} = 0$ if $i \geq 3$ or $j \geq 3$, which implies the spectral sequence degenerates at the $E_3$ page.  In particular, we get a surjection
$$\Ext_G^3(\Ind_B^G(\sigma),\pi) \longtwoheadrightarrow \Ext_T^1(\sigma, \textnormal{R}^2\cR_B^G(\pi)).$$
By equations \eqref{dimsigmaR2} and \eqref{Ext3dim=1}, the left-hand side has dimension 1 while the right-hand side has dimension 2, and we arrive at a contradiction.
\end{proof}

It remains to calculate $\textnormal{R}^1\cR_B^G(\pi)$.  We proceed as follows.  Using the values of $\textnormal{R}^n\cR_B^G(\pi)$ already computed, the spectral sequence \eqref{ord1} yields the following exact sequence:
\begin{equation}
\label{7term1}
\begin{tikzcd}
0\ar[r]  & \Ext_T^1(\sigma, \cR_B^G(\pi)) \ar[r, "e_1"] & \Ext_G^1(\Ind_B^G(\sigma), \pi) \ar[r] \arrow[d, phantom, ""{coordinate, name=Z}]   & \Hom_T(\sigma, \textnormal{R}^1\cR_B^G(\pi)) \arrow[dll, rounded corners=8pt, to path= { -- ([xshift=3ex]\tikztostart.east) |- (Z) [near end]\tikztonodes -| ([xshift=-3ex]\tikztotarget.west) -- (\tikztotarget)}] & \\
 & \Ext^2_T(\sigma, \cR_B^G(\pi)) \ar[r, "e_2"] & \Ext^2_G(\Ind_B^G(\sigma), \pi) \ar[r]   & \Ext^1_T(\sigma, \textnormal{R}^1\cR_B^G(\pi)) \ar[r] & 0 
\end{tikzcd}
\end{equation}
Here, the maps $e_i$ are the edge maps of the spectral sequence \eqref{ord1}.  Similarly, using the fact that $\textnormal{H}^j\ord_{B^-}^G(\pi) = 0$ for $j \geq 2$ (\cite[Prop. 3.6.1]{emerton:ord2}) and that $\pi$ is locally admissible, the spectral sequence \eqref{ord3} for $\sigma$ locally admissible gives
\begin{equation}
\label{7term2}
\begin{tikzcd}
0\ar[r]  & \Ext_T^1(\sigma, \ord_{B^-}^G(\pi)) \ar[r, "e_1^{\textnormal{ladm}}"] & \Ext_G^1(\Ind_B^G(\sigma), \pi) \ar[r] \arrow[d, phantom, ""{coordinate, name=Z}]   & \Hom_T(\sigma, \textnormal{H}^1\ord_{B^-}^G(\pi)) \arrow[dll, rounded corners=8pt, to path= { -- ([xshift=3ex]\tikztostart.east) |- (Z) [near end]\tikztonodes -| ([xshift=-3ex]\tikztotarget.west) -- (\tikztotarget)}] & \\
 & \Ext^2_T(\sigma, \ord_{B^-}^G(\pi)) \ar[r, "e_2^{\textnormal{ladm}}"] & \Ext^2_G(\Ind_B^G(\sigma), \pi) \ar[r]   & \Ext^1_T(\sigma, \textnormal{H}^1\ord_{B^-}^G(\pi)) \ar[r] & 0 
\end{tikzcd}
\end{equation}

Let us now fix a locally admissible $T$-representation $\sigma$, and recall that $\pi = \Ind_B^G(\chi)$.  Consider the composite map
$$\Ext_T^i(\sigma,\chi) \longrightarrow \Ext_T^i\big(\sigma,\cR_B^G(\Ind_B^G(\chi))\big) \stackrel{e_i}{\longrightarrow} \Ext_G^i(\Ind_B^G(\sigma), \Ind_B^G(\chi)),$$
where the first map is induced by the unit $\chi \longrightarrow \cR_B^G(\Ind_B^G(\chi))$ of the adjunction $(\Ind_B^G, \cR_B^G)$.  A straightforward exercise in homological algebra shows that this composite map is, up to a sign, equal to the map obtained by applying the exact functor $\Ind_B^G$ to a Yoneda extension.  A similar remark holds for the adjunction $(\Ind_B^G, \ord_{B^-}^G)$ defined on the locally admissible categories.

Suppose further that $\sigma$ has finite length; this implies that all Ext spaces appearing in the exact sequences \eqref{7term1} and \eqref{7term2} have finite dimension.  By \cite[Thm. 5.3]{vigneras:rightadj} and \cite[Prop. 4.3.4]{emerton:ord1}, the units of both the adjunctions $(\Ind_B^G, \cR_B^G)$ and $(\Ind_B^G, \ord_{B^-}^G)$ are isomorphisms, which implies that the domain of $e_2$ is identified with the domain of $e_2^{\textnormal{ladm}}$ (both of which are isomorphic to $\Ext^2_T(\sigma,\chi)$).  Therefore, the paragraph above implies that the image of the edge map $e_2$ in the exact sequence \eqref{7term1} has the same dimension as the image of the edge map $e_2^{\textnormal{ladm}}$ in the exact sequence \eqref{7term2}.  In particular, by dimension counting this implies 
$$\dim_C\big(\Ext^1_T(\sigma, \textnormal{R}^1\cR_B^G(\pi))\big) = \dim_C\big(\Ext^1_T(\sigma, \textnormal{H}^1\ord_{B^-}^G(\pi))\big);$$
the two exact sequences \eqref{7term1} and \eqref{7term2} then give
\begin{equation}
\label{dimhomR1}
\dim_C\big(\Hom_T(\sigma, \textnormal{R}^1\cR_B^G(\pi))\big) = \dim_C\big(\Hom_T(\sigma, \textnormal{H}^1\ord_{B^-}^G(\pi))\big) = \dim_C\big(\Hom_T(\sigma, \chi^{-1}\overline{\alpha})\big),
\end{equation}
where the last equality follows from \cite[Cor. 3.3.8(ii)]{hauseux}.

\begin{lemma}
We have $\textnormal{R}^1\cR_B^G(\pi) = \chi^{-1}\overline{\alpha}$.  
\end{lemma}

\begin{proof}
Taking $\sigma = \chi^{-1}\overline{\alpha}$ in \eqref{dimhomR1} shows that $\chi^{-1}\overline{\alpha} \longhookrightarrow \textnormal{R}^1\cR_B^G(\pi)$.  Now, assume by contradiction that $\dim_C(\textnormal{R}^1\cR_B^G(\pi)) \geq 2$, and let $\sigma$ denote the second step of the socle filtration of $\textnormal{R}^1\cR_B^G(\pi)$ (as in Step 2 of the proof of Lemma \ref{lemma-R2=0}).  Then $\sigma$ is a nonsplit extension of $\chi^{-1}\overline{\alpha}$ by itself, and by construction, we have
$$\dim_C\big(\Hom_T(\sigma,\textnormal{R}^1\cR_B^G(\pi))\big) = 2,\qquad \dim_C\big(\Hom_T(\sigma,\chi^{-1}\overline{\alpha})\big) = 1.$$
However, this contradicts equation \eqref{dimhomR1}.  
\end{proof}

Putting everything together gives
\begin{equation}
\label{result-ps}
\boxed{\textnormal{R}^n\cR_B^G(\Ind_B^G(\chi)) = \begin{cases} \chi & \textnormal{if $n = 0$}, \\ \chi^{-1}\overline{\alpha} & \textnormal{if $n = 1$}, \\ 0 & \textnormal{if $n \geq 2$}. \end{cases}}
\end{equation}

\subsection{Steinberg}
Suppose now that $\pi = \textnormal{St} := \Ind_B^G(\mathbf{1}_T)/\mathbf{1}_G$ is the Steinberg representation.  By \cite[\S 5.4.3]{koziol:functorial}, we have
$$\textnormal{H}^n(I_1,\textnormal{St}) = \begin{cases} \chi_{\sign} & \textnormal{if $n = 0$,} \\ \Ind_{\cH_T}^{\cH}(\mathbf{1}_T) & \textnormal{if $n = 1$,} \\ \chi_{\triv} & \textnormal{if $n = 2$,} \\ 0 & \textnormal{if $n \geq 3$.} \end{cases}$$
(For the definitions of $\chi_{\sign}$ and $\chi_{\triv}$, see \cite[Rmks. 2.23 and 2.24(1)]{olliviervigneras}.)  Thus, by \cite[Thm. 5.20]{abe:inductions}, we have
$$\cR_{\cH_T}^{\cH}(\textnormal{H}^n(I_1,\textnormal{St})) = \begin{cases} \mathbf{1}_T & \textnormal{if $n = 0, 1$,}  \\  0 & \textnormal{if $n \geq 2$.} \end{cases}$$

By \cite[Cor. 6.5]{AHV}, we have 
$$\cR_B^G(\textnormal{St}) = \mathbf{1}_T.$$ 
Using this fact, the calculation of $\cR_{\cH_T}^{\cH}(\textnormal{H}^1(I_1,\textnormal{St}))$ above, and the short exact sequence \eqref{deg2} for $n = 1$, we get
$$0 \longrightarrow \textnormal{H}^1(T_1, \mathbf{1}_T) \cong \mathbf{1}_T \longrightarrow \mathbf{1}_T \longrightarrow (\textnormal{R}^1\cR^G_B(\textnormal{St}))^{T_1} \longrightarrow 0.$$
Therefore, we have $(\textnormal{R}^1\cR^G_B(\pi))^{T_1} = 0$, which implies
$$\textnormal{R}^1\cR^G_B(\textnormal{St}) = 0.$$
Finally, using the short exact sequence \eqref{deg2} for $n \geq 2$ and the calculation of $\cR_{\cH_T}^{\cH}(\textnormal{H}^n(I_1,\textnormal{St}))$ shows that $(\textnormal{R}^n\cR^G_B(\pi))^{T_1} = 0$.  Putting everything together gives
$$\boxed{\textnormal{R}^n\cR^G_B(\textnormal{St}) = \begin{cases} \mathbf{1}_T & \textnormal{if $n = 0$}, \\ 0 & \textnormal{if $n \geq 1$}. \end{cases}}$$

\subsection{Trivial representation}
Suppose next that $\pi = \mathbf{1}_G$ is the trivial $G$-representation.  By \cite[Cor. 6.5]{AHV}, we have 
$$\cR_B^G(\mathbf{1}_G) = 0.$$  
To compute higher derived functors, we use the short exact sequence
$$0 \longrightarrow \mathbf{1}_G \longrightarrow \Ind_B^G(\mathbf{1}_T) \longrightarrow \textnormal{St} \longrightarrow 0.$$
Applying the left-exact functor $\cR_B^G$ to the above gives an exact sequence
$$0 \longrightarrow \cR_B^G(\mathbf{1}_G) = 0 \longrightarrow \cR_B^G(\Ind_B^G(\mathbf{1}_T)) \cong \mathbf{1}_T \longrightarrow \cR_B^G(\textnormal{St}) \cong \mathbf{1}_T \longrightarrow 0.$$
Since $\textnormal{R}^n\cR_B^G(\textnormal{St}) = 0$ for $n \geq 1$, the long exact sequence for higher derived functors implies $\textnormal{R}^n\cR_B^G(\mathbf{1}_G) \cong \textnormal{R}^n\cR_B^G(\Ind_B^G(\mathbf{1}_T))$ for all $n \geq 1$.  Thus, using equation \eqref{result-ps}, we conclude
$$\boxed{\textnormal{R}^n\cR_B^G(\mathbf{1}_G) = \begin{cases} 0 & \textnormal{if $n = 0$}, \\ \overline{\alpha} & \textnormal{if $n = 1$}, \\ 0 & \textnormal{if $n \geq 2$}. \end{cases}}$$

\subsection{Supersingular representations}
Finally, suppose that $\pi$ is an absolutely irreducible supersingular $G$-representation.  Then the $\cH$-modules $\textnormal{H}^n(I_1, \pi)$ are supersingular for all $n \geq 0$ (when $C$ is finite, one can use \cite[\S 5.4.4]{koziol:functorial}; otherwise, see \cite[Cor. 8.12]{ollivierschneider:pro-p-SL2}).  Consequently, by \cite[Thm. 5.20]{abe:inductions}, we have $\cR_{\cH_T}^{\cH}(\textnormal{H}^n(I_1,\pi)) = 0$ for all $n \geq 0$, and equations \eqref{deg1} and \eqref{deg2} imply $(\textnormal{R}^n\cR^G_B(\pi))^{T_1} = 0$ for all $n \geq 0$.  Thus, we conclude
$$\boxed{\textnormal{R}^n\cR^G_B(\pi) = 0 \qquad \textnormal{for all $n\geq 0$.}}$$

\bibliographystyle{amsalpha}
\bibliography{refs}

\providecommand{\bysame}{\leavevmode\hbox to3em{\hrulefill}\thinspace}
\providecommand{\MR}{\relax\ifhmode\unskip\space\fi MR }
\providecommand{\MRhref}[2]{%
  \href{http://www.ams.org/mathscinet-getitem?mr=#1}{#2}
}
\providecommand{\href}[2]{#2}
\begin{thebibliography}{DHKM22}

\bibitem[Abd14]{abdellatif}
Ramla Abdellatif, \emph{Classification des repr\'esentations modulo {$p$} de
  {${\rm SL}(2,F)$}}, Bull. Soc. Math. France \textbf{142} (2014), no.~3,
  537--589. \MR{3295722}

\bibitem[Abe19a]{abe:involutions}
Noriyuki Abe, \emph{Involutions on pro-{$p$}-{I}wahori {H}ecke algebras},
  Represent. Theory \textbf{23} (2019), 57--87. \MR{3902325}

\bibitem[Abe19b]{abe:inductions}
\bysame, \emph{Parabolic inductions for pro-{$p$}-{I}wahori {H}ecke algebras},
  Adv. Math. \textbf{355} (2019), 106776, 63. \MR{3996728}

\bibitem[AHV19]{AHV}
N.~Abe, G.~Henniart, and M.-F. Vign\'{e}ras, \emph{Modulo {$p$} representations
  of reductive {$p$}-adic groups: functorial properties}, Trans. Amer. Math.
  Soc. \textbf{371} (2019), no.~12, 8297--8337. \MR{3955548}

\bibitem[Ber87]{bernstein:secondadjoint}
Joseph Bernstein, \emph{Second adjointness for reductive {\textit{p}}-adic
  groups}, Unpublished notes, available at
  \url{http://www.math.tau.ac.il/~bernstei/Unpublished_texts/unpublished_texts/Bernstein87-second-adj-from-chicago.pdf}
  (1987).

\bibitem[BL94]{barthellivne}
L.~Barthel and R.~Livn\'{e}, \emph{Irreducible modular representations of
  {${\rm GL}_2$} of a local field}, Duke Math. J. \textbf{75} (1994), no.~2,
  261--292. \MR{1290194}

\bibitem[Bus01]{bushnell}
Colin~J. Bushnell, \emph{Representations of reductive {$p$}-adic groups:
  localization of {H}ecke algebras and applications}, J. London Math. Soc. (2)
  \textbf{63} (2001), no.~2, 364--386. \MR{1810135}

\bibitem[Dat09]{dat}
Jean-Francois Dat, \emph{Finitude pour les repr\'{e}sentations lisses de
  groupes {$p$}-adiques}, J. Inst. Math. Jussieu \textbf{8} (2009), no.~2,
  261--333. \MR{2485794}

\bibitem[DHKM22]{DHKM:heckealgs}
Jean-Francois {Dat}, David {Helm}, Robert {Kurinczuk}, and Gilbert {Moss},
  \emph{{Finiteness for Hecke algebras of $p$-adic groups}}, available at
  \href{https://arxiv.org/abs/2203.04929}{arXiv:2203.04929} (2022).

\bibitem[Eme10a]{emerton:ord1}
Matthew Emerton, \emph{Ordinary parts of admissible representations of
  {$p$}-adic reductive groups {I}. {D}efinition and first properties},
  Ast\'{e}risque (2010), no.~331, 355--402. \MR{2667882}

\bibitem[Eme10b]{emerton:ord2}
\bysame, \emph{Ordinary parts of admissible representations of {$p$}-adic
  reductive groups {II}. {D}erived functors}, Ast\'erisque (2010), no.~331,
  403--459. \MR{2667883}

\bibitem[EP10]{emertonpaskunas}
Matthew Emerton and Vytautas Pa\v{s}k\={u}nas, \emph{On the effaceability of
  certain {$\delta$}-functors}, Ast\'erisque (2010), no.~331, 461--469.
  \MR{2667892}

\bibitem[Har16]{harris:specs}
Michael Harris, \emph{Speculations on the {${\rm mod}\, p$} representation
  theory of {$p$}-adic groups}, Ann. Fac. Sci. Toulouse Math. (6) \textbf{25}
  (2016), no.~2-3, 403--418. \MR{3530163}

\bibitem[Hau18]{hauseux}
Julien Hauseux, \emph{Parabolic induction and extensions}, Algebra Number
  Theory \textbf{12} (2018), no.~4, 779--831. \MR{3830204}

\bibitem[HV19]{henniartvigneras:descentp}
G.~Henniart and M.-F. Vign\'{e}ras, \emph{Representations of a {$p$}-adic group
  in characteristic {$p$}}, Representations of reductive groups, Proc. Sympos.
  Pure Math., vol. 101, Amer. Math. Soc., Providence, RI, 2019, pp.~171--210.
  \MR{3930018}

\bibitem[Koz21]{koziol:functorial}
Karol Kozio\l, \emph{Functorial properties of pro-{$p$}-{I}wahori cohomology},
  J. Lond. Math. Soc. (2) \textbf{104} (2021), no.~4, 1572--1614. \MR{4339945}

\bibitem[OS21]{ollivierschneider:pro-p-SL2}
Rachel {Ollivier} and Peter {Schneider}, \emph{{On the pro-$p$ Iwahori Hecke
  Ext-algebra of ${\rm SL}_2(\mathbb Q_p)$}}, available at
  \href{https://arxiv.org/abs/2104.13422}{arXiv:2104.13422} (2021).

\bibitem[OV18]{olliviervigneras}
Rachel Ollivier and Marie-France Vign\'{e}ras, \emph{Parabolic induction in
  characteristic {$p$}}, Selecta Math. (N.S.) \textbf{24} (2018), no.~5,
  3973--4039. \MR{3874689}

\bibitem[Pa{\v{s}}04]{paskunas:diags}
Vytautas Pa{\v{s}}k{\=u}nas, \emph{Coefficient systems and supersingular
  representations of {${\rm GL}_2(F)$}}, M\'{e}m. Soc. Math. Fr. (N.S.) (2004),
  no.~99, vi+84. \MR{2128381}

\bibitem[Pa{\v{s}}10]{paskunas:exts}
\bysame, \emph{Extensions for supersingular representations of {${\rm
  GL}_2(\Bbb Q_p)$}}, Ast\'erisque (2010), no.~331, 317--353. \MR{2667891}

\bibitem[Pa{\v{s}}13]{paskunas:montrealfunctor}
\bysame, \emph{The image of {C}olmez's {M}ontreal functor}, Publ. Math. Inst.
  Hautes \'Etudes Sci. \textbf{118} (2013), 1--191. \MR{3150248}

\bibitem[Sch15]{schneider:dga}
Peter Schneider, \emph{Smooth representations and {H}ecke modules in
  characteristic p}, Pacific J. Math. \textbf{279} (2015), no.~1-2, 447--464.
  \MR{3437786}

\bibitem[SS22]{scherotzkeschneider}
Sarah Scherotzke and Peter Schneider, \emph{Derived parabolic induction}, Bull.
  Lond. Math. Soc. \textbf{54} (2022), no.~1, 264--274. \MR{4408619}

\bibitem[Vig16a]{vigneras:hecke1}
Marie-France Vign{\'e}ras, \emph{The pro-{$p$}-{I}wahori {H}ecke algebra of a
  reductive {$p$}-adic group {I}}, Compos. Math. \textbf{152} (2016), no.~4,
  693--753. \MR{3484112}

\bibitem[Vig16b]{vigneras:rightadj}
\bysame, \emph{The right adjoint of the parabolic induction}, Arbeitstagung
  {B}onn 2013, Progr. Math., vol. 319, Birkh\"{a}user/Springer, Cham, 2016,
  pp.~405--425. \MR{3618059}

\end{thebibliography}

\end{document}